\documentclass[12pt]{article}
\usepackage{amsmath}
\usepackage{amssymb} 
\usepackage{latexsym}
\usepackage{theorem}
\usepackage{euscript}

\newcommand{\msf}[1]{\mathsf {#1}}

\newcommand{\mcal}[1]{{\mathcal {#1}}} 

\newtheorem{theorem}{Theorem}  [section]
\newtheorem{lemma}[theorem]{Lemma}
\newtheorem{corollary}[theorem]{Corollary}
\newtheorem{proposition}[theorem]{Proposition}
\newtheorem{hypothesis}[theorem]{Hypothesis}
{\theorembodyfont{\rmfamily}  	
                                \newtheorem{remark}[theorem]{Remark}
                                
				\newtheorem{question}[theorem]{Question}

                                \newtheorem{definition}[theorem]{Definition}
                                
}
{\theoremstyle{break}	
{\theorembodyfont{\rmfamily} 
				\newtheorem{example}[theorem]{Example}
				
}}
{\end{enumerate}}

\newenvironment{proof}[1]{\smallskip \noindent {\bf #1}}{\qed\smallskip}

\def\qed{\ifhmode\unskip\nobreak\fi\ifmmode\ifinner\else\hskip5pt\fi\fi
 \hfill\hbox{\hskip5pt\vrule width4pt height6pt depth1.5pt\hskip1pt}}
 
\newcommand{\ov}[1]{\ensuremath{\overline{#1}}}

\newcommand{\Int}{\ensuremath{\operatorname{\mathsf{Int}}}}

\newcommand{\RR}{\ensuremath{{\mathbb R}}}     
\newcommand{\R}[1]{\ensuremath{{\mathbb R}^{#1}}} 

\newcommand{\Rp}{\ensuremath{{\mathbb R}_+}}   

\newcommand{\ZZ}{\ensuremath{{\mathbb Z}}}	   
\newcommand{\Np}{\ensuremath{{\mathbb N}_{+}}}  
\renewcommand{\S}[1]{\ensuremath{{\bf S}^{#1}}} 
\newcommand{\D}[1]{\ensuremath{{\bf D}^{#1}}} 

\newcommand{\om}[1]{\ensuremath{\omega(#1)}}

\newcommand {\pd}[1] {\ensuremath{\frac{\partial}{\partial #1}}}
\newcommand{\pde}[2]
	{\ensuremath{\frac{\partial #1}{\partial #2}}}	
\newcommand{\ode}[2]{\ensuremath{\frac{d#1}{d#2}}}    	

\def\d#1dt{\frac{d#1}{dt}}    
\newcommand{\eps}{\ensuremath{\epsilon}}
\newcommand{\lam}{\ensuremath{\lambda}}
\newcommand{\Lam}{\ensuremath{\Lambda}}

\newcommand{\gam}{\ensuremath{\gamma}}

\newcommand{\Fix}[1]{\ensuremath{\operatorname{\mathsf {Fix}}(#1)}}
\newcommand{\co}{\colon\thinspace} 

\def\empty{\varnothing}

\renewcommand{\gg}{\ensuremath{\mathfrak g}}

\newcommand{\mm}{\ensuremath{\mathfrak m}}
\newcommand{\nn}{\ensuremath{\mathfrak n}}

\renewcommand{\ss}{\ensuremath{\mathfrak {s}}}
\newcommand{\ttt}{\ensuremath{\mathfrak t}}

\newcommand{\p}{\ensuremath{\partial}}

\usepackage{txfonts} 
\usepackage[active]{srcltx}

\begin{document}      
 
\def\empty{\varnothing}
\newcommand{\V}{\mcal V}
\renewcommand{\om}{\omega}
\reversemarginpar 	
 \pagestyle{plain}

      \def\mylabel#1{\label{#1}} 
	
\newcommand{\Z}[1]{\ensuremath{{\msf  Z ( #1)}}}

\renewcommand{\ss}  {\ensuremath{{\mathfrak {s}}}}
	
\renewcommand{\dim} {\ensuremath{{\mathsf {dim}\,}}}

\newcommand{\bd}[1]{\ensuremath{{\msf  {bd}  (#1)}}}

 \title{\bf Common zeros of inward  vector fields on
    surfaces}

\author{\Large  Morris W. Hirsch\thanks
{This research was supported in part by National Science Foundation
  grant DMS-9802182.}\\
 Mathematics Department\\
University of Wisconsin at Madison\\
University of California at Berkeley}
\maketitle
%




\begin{abstract} 
A vector field $X$ on a manifold  $M$ with possibly nonempty boundary is
{\em inward} if it generates a unique local semiflow $\Phi^X$.    A compact
relatively open set $K$ in the zero set $\Z X$ is a {\em  block}.  
The Poincar\'e-Hopf index is extended to blocks $K\subset \Z X$ where
$X$ is inward and $K$ 
may meet $\p M$.  A block with nonzero index is {\em
  essential}.  

Let $X, Y$ be inward $C^1$ vector fields on surface $M$ such that
 $[X, Y]\wedge X=0$ and let $K$ be an essential block of zeros for $X$.
Among the main results are that $Y$ has a zero in $K$ if $X$ and $Y$
are analytic, or $Y$ is $C^2$ and $\Phi^Y$ preserves area. 
Applications are made to actions of Lie algebras and groups.
\end{abstract}

\tableofcontents
\section{Introduction}   \mylabel{sec:intro}
Let $M$ be an $n$-dimensional manifold with boundary $\p M$ and $X$ a
vector field on $M$, whose value at $p$ is denoted by $X_p$.  The {\em
  zero set} of $X$ is
\[
 \Z X:= \{p\in M\co X_p=0\}.
\]
The set of common zeros for a  set $\ss$  of vector fields is
\[\textstyle
  \Z \ss :=\bigcap_{X\in \ss}\Z X. 
\] 

A {\em block} for $X$ is a compact, relatively open subset $K\subset
\Z X$.  This means $K$ lies in a precompact open set $U\subset M$
whose topological boundary $\msf{bd}(U)$ contains no zeros of $X$.  We
say that $U$ is {\em isolating} for $(X,K)$, and for $X$ when $K:=\Z X
\cap U$.   When $M$ is compact, $\Z
X$ is a block for $X$ with $M$ as an isolating neighborhood.

\begin{definition}              \mylabel{th:defph}
If $p\in M\verb=\=\p M$ is an isolated zero of $X$, the {\em
 index of $X$ at $p$}, denoted by $\msf i_p X$,  is  the degree of the map of
the unit $(n-1)$-sphere 
\[ 
  \S{n-1}\to\S {n-1}, \quad z\mapsto \frac{\hat X(z)}{\|\hat X (z)\|},
\]
where $\hat X$ is the representative of $X$ in an arbitrary coordinate
system centered at $p$.  When $U$ is isolating for $(X, K)$ and
disjoint from $\Z X \cap \p M$, the {\em Poincar\'e-Hopf index} of $X$
at $K$ is
\[
  \msf i^{\rm PH}_K (X) :=\sum_p\msf i_pY, \quad (p\in \Z Y \cap U)
\]
where $Y$ is any vector field on $M$ such that $\Z Y \cap \ov U$ is
finite, and there is a homotopy of vector fields 
$\left\{X^t\right\}_{0\le t\le 1}$ from $X^0=X$ to $X^1=Y$ such that 
 $\bigcup_{t}]\Z {X^t} \cap U$ is compact.\footnote{
The Poincar\'e-Hopf index goes back to  Poincar\'e
\cite{Poincare85} and  Hopf \cite{Hopf25}.  It is usually defined only when
$M$ is compact and $K=\Z X$.  The more general definition above is adapted
from Bonatti \cite{Bonatti92}.}

The block $K$  is {\em essential} for $X$  if\, $\msf i^{\rm  PH}_K (X)\ne 0$.
\end{definition}

Christian Bonatti \cite{Bonatti92} proved the following remarkable result:
\begin{theorem}[{\sc  Bonatti}]        \mylabel{th:bonatti}
Assume $\dim M\le 4$ and $\p M=\empty$.  If $X,Y$ are analytic vector
fields on $M$ such that $[X, Y]= 0$, then $\Z Y$ meets every essential
block of zeros for $X$.\footnote{
Bonatti assumes $\dim M = 3$ or $4$, but the conclusion for $\dim M
\le 2$ follows easily:  If $\dim M=2$, identify $M$ with
$M\times\{0\}\subset M\times \RR$ and apply Bonatti's theorem to the
vector fields $\left(X, x\pde{~} {x}\right)$ and $\left(Y, x\pde{~} {x}\right)$ on $M\times
\RR$.   For $\dim M=1$ there is a simple direct proof.} 
\end{theorem}

Our main results, Theorems \ref{th:main} and \ref{th:mainB}, reach
similar conclusions for surfaces $M$ which may have nonsmooth
boundaries, and certain pairs of vector fields that generate local
semiflows on $M$, including cases where the fields are not analytic
and do not not commute. Applications are made to actions of Lie
algebras and Lie groups.

Next we define terms (postponing some details), state the main
theorems and apply them to Lie actions.  After sections on dynamics
and index functions, the main results are proved in Section
\ref{sec:mainproofs}.

\subsection{Terminology}   
$\ZZ$ denotes the integers, $\Np$ the positive integers, $\RR$ the
real numbers, and $\Rp$ the closed half line $[0,\infty)$.
Maps are continuous, and 
manifolds are real, smooth and metrizable, unless otherwise noted.
The set of fixed points of a map $f$ is $\Fix f$.

The following assumptions are always in force:

\begin{hypothesis}              \mylabel{th:hypmain}
$\tilde M$ is an analytic $n$-manifold with empty boundary. 
$M\subset \tilde M$ is a connected topological $n$-manifold.\footnote{
The only role of  $\tilde M$ is to permit a simple definition of
smooth maps on $M$.  Its global topology is never
used, and in any discussion $\tilde M$ can be replaced by any
smaller open neighborhood of $M$.
If $n>2$ then $M$ might not be smoothable, as shown by a construction due
  to Kirby \cite{Kirby12}:  Let $P$ be a nonsmoothable closed
  $4$-manifold  (Freedman
  \cite{Freedman82}).  Let $D\subset P$ be a compact $4$-disk. Then
  $M:=P\setminus \Int D$ is not smoothable, for otherwise $\p M$ would
  be diffeomorphic to $\S 3$ and $P$ could be smoothed by gluing $\D
  4$ to $M$.
  Define $\tilde M$ as the connected sum of $P$ with a nontrivial $\S
  2$-bundle over $\S 2$.  Then $\tilde M$ is smoothable and contains
  $M$ (compare Friedl {\em et al.\ }\cite {Friedl07}).}
\end{hypothesis}
We call $M$ an {\em analytic manifold} when $\p M$ is an analytic
submanifold of $\tilde M$. The tangent vector bundle of $\tilde M$ is 
$T\tilde M$, whose fibre over $p$
is the vector space $T_p\tilde M$.  The restriction of $T\tilde M$ to a
subset $S\subset \tilde M$ is the vector bundle $T_S\tilde M$.  We set
$TM:=T_M\tilde M$.

A map $f$ sending a set $S\subset M$ into a smooth manifold $N$ is
called $C^r$ if it extends to a map $\tilde f$, defined on an open
subset of $\tilde M$, that is $C^r$ in the usual sense.  Here
$r\in\Np\cup \{\infty, \omega\}$, and $C^\omega$ means analytic.  If
$f$ is $C^1$ and $S$ has dense interior in $M$, the tangent map
$T\tilde f\co T\tilde M\to T N$ restricts to a bundle map $Tf\co T_SM
\to TN$ determined by $f$.

A {\em vector field} on $S$ is a section $X\co S\to T_S M$,
whose value at $p$ is denoted by $X_p$.  The set of these vector
fields is a linear space $\V (S)$.  The linear subspaces $\V^r (S)$
and $\V^{\msf{L}}(M)$, comprising $C^r$ and locally Lipschitz fields
respectively, are given the compact-open topology (uniform convergence
on compact sets).

$X$ and $Y$ always denote vector fields on $M$.  When $X$ is $C^r$,
\ $\tilde X$ denotes an extension of $X$ to a $C^r$ vector field on an
open set $W\subset \tilde M$. 

The {\em Lie bracket} of $X, Y \in \V^1 (M)$ is the restriction to $M$
of $[\tilde X, \tilde Y]$. This operation makes $\V^\omega (M)$ and
$\V^\infty (M)$ into Lie algebras.

$X\wedge Y$ denotes the tensor field of exterior 2-vectors $p\mapsto
X_p\wedge Y_p \in \Lam^2(T_pM)$.  Evidently $X\wedge Y=0$ \ iff \ $X_p$ 
and $Y_p$ are linearly dependent at all $p\in M$.
 
\paragraph{Inward vector fields}
A tangent vector to $M$ at $p$ is {\em inward} if it is the tangent at
$p$ to a smooth curve in $M$ through $p$.  The set of inward vectors
at $p$ is $T^{\msf {in}}_pM$. A vector field $X$ is {\em inward} if
$X(M)\subset T^{\msf {in}} (M)$, and there is a unique local semiflow
$\Phi^X:=\left\{\Phi^X_t\right\}_{t\in \Rp}$ on $M$ whose trajectories are the
maximally defined solutions to the initial value problems

\[ \ode  y t = X(y), \quad y(0)=p, \qquad p\in M, \quad t \ge 0.\]

The set of inward vector fields is $\V_{\msf {in}} (M)$.  When $\p M$
is a $C^1$ submanifold of $\tilde M$, it can be shown that $X$ is
inward iff $X (M)\subset T^{\msf {in}} (M)$.

 Define
\[
 \V^r_{\msf {in}} (M):= \V_{\msf {in}} (M)\cap\V^r (M), \qquad 
\V^{\msf{L}}_{\msf {in}}(M):= \V_{\msf {in}} (M)\cap  \V^{\msf{L}}(M).
\]
Proposition \ref{th:convex}
  shows these sets are  convex cones in $\V (M)$.

\paragraph{The vector field index and essential blocks of zeros}
Let $K$ be a block of zeros for $X\in \V_{\msf{in}} (M)$, and
$U\subset M$ an isolating neighborhood for $(X, K)$.   
 The
{\em vector field index}
\[
  \msf i_K (X):=\msf i (X, U)\in \ZZ
\]
is defined in Section \ref{sec:index} as the fixed point index of the
map \ $\Phi^X_t|U\co U\to M$, for any $t >0$ so small that the compact
set $\ov U$ lies in the domain of $\Phi^X_t$.

The block $K$ is
{\em essential} (for $X$) when $\msf i_K (X)\ne 0$.    A version of the
Poincar\'e-Hopf theorem implies  $K$ is essential if it is
an attractor for $\Phi^X$ and has nonzero Euler characteristic $\chi
(K)$.  

\subsection{Statement of results}   
In the next two theorems, besides Hypothesis \ref{th:hypmain} we
assume:

\begin{hypothesis}              \mylabel{th:hyp2}
{~}
\begin{itemize}
\item $M$ and $\tilde M$ are surfaces,

\item  $X$ and $Y$ are $C^1$ inward vector fields on $M$,

\item $K\subset M$ is an essential block of zeros for $X$,

\item  $[X, Y]\wedge X=0$.
\end{itemize}
\end{hypothesis}
The last  condition has the following dynamical significance
(Proposition \ref{th:wedge}):   
\begin{itemize}

\item {\em $\Phi^Y$ permutes
integral curves of $X$.}
\end{itemize}
This implies:
\begin{itemize}

\item {\em if  $q=\Phi^Y_t (p)$ then \
$ X_q=\lam\cdot T\Phi^Y_t(X_p)$ \  for some $ \lam  >0$,}

\item {\em  $\Z X$ is positively invariant under $\Phi^Y$.} (See
  Definition  \ref{th:posinv}.)
\end{itemize}
 
A {\em cycle} for $Y$, or  a $Y$-{\em cycle}, is a 
 periodic orbit of $\Phi^Y$ that is not a fixed point.   
\begin{theorem}         \mylabel{th:main}
Assume {\em Hypothesis \ref{th:hyp2}}.  Each of the following
conditions implies $\Z Y\cap 
K\ne\varnothing$: 
\begin{description}

\item[(a)] $X$ and $Y$ are analytic.

\item[(b)] Every neighborhood of $K$ contains an open neighborhood
  whose boundary is a nonempty union of finitely many  $Y$-cycles. 
 
\end{description}
\end{theorem}
When  $[X, Y]=0$ this extends Bonatti's Theorem to surfaces with
nonempty boundaries.  The case $[X, Y]= cX, \, c\in\RR$ yields
applications to actions of Lie algebras and Lie groups.

\begin{example}         \mylabel{th:exlima}
In his pioneering paper \cite{Lima64}, E. Lima constructs vector
fields $X, Y$ on the closed disk $\D 2$, tangent to $\p \D 2$ and
generating unique flows, such that $[X, Y]= X$ and $\Z X\cap\Z
Y=\empty$ (see Remark \ref{th:remlima}).  Such fields can be
$C^\infty$ (M. Belliart \& I. Liousse \cite{BL96}, F.-J Turiel
\cite{Turiel03}).  
The unique block of
zeros for $X$ is $\Z X=\p \D 2$, which is essential because $\chi (\D
2)\ne\varnothing$, but $\Z Y$ is a point in the interior of $\D 2$.   
This shows that the conclusion of Theorem \ref{th:main}(a) can fail
when $X$ and $Y$ are not analytic.
The  flows of $X$ and $Y$  generate an effective,
fixed-point free action by a connected,  solvable nonabelian Lie
group.

\end{example}

A local semiflow on a surface $M$ {\em preserves area} if it
preserves a Borel measure on $M$ that is positive and 
finite on nonempty precompact sets.  
\begin{theorem}         \mylabel{th:mainB}
Assume {\em Hypothesis \ref{th:hyp2}}.  If  $\Phi^Y$ preserves area,
each of the following 
conditions implies $\Z Y \cap K\ne\varnothing$:

\begin{description}

\item[(i)] $K$  contains a $Y$-cycle, 

\item[(ii)] $Y$ is $C^2$, 

\item[(iii)] $K$ has a planar neighborhood in $M$. 

\end{description}
\end{theorem}
\begin{definition}              \mylabel{th:defW}
For $X\in \V^\om_{\msf {in}} (M)$ define
\[
 \msf W(X):=\{Y\in\V^\om_{\msf {in}} (M) \co[X, Y]\wedge X=0\},
\]
which is the set of inward analytic vector fields on $M$ whose local
semiflows permute integral curves of $X$.  Propositions \ref{th:wedge}
and \ref{th:convex} imply $\msf W(X)$ is a convex cone that is closed
under Lie brackets, and  a subalgebra of $\V^\om (M)$
if $M$ is an analytic manifold without boundary.
\end{definition}
\begin{theorem}         \mylabel{th:mainC}
Assume Hypothesis \ref{th:hyp2} holds.  If   $X$ is analytic and
$\p M$ is  an analytic
subset of $\tilde M$, then 
$   \Z{\msf W (X)}\cap K\ne\varnothing$.
\end{theorem}

\paragraph{Actions of Lie algebras and Lie groups}

\begin{theorem}         \mylabel{th:liealg}
Let $M$ be an analytic surface and $\gg$ a Lie algebra (perhaps
infinite dimensional) of analytic vector fields on $M$ that are
tangent to $\p M$. Assume $X\in \gg$ spans a nontrivial ideal.
Then:
\begin{description}

\item[(a)] $\Z \gg$ meets every essential block of zeros for $X$.

\item[(b)] If $M$ is compact and $\chi (M)\ne 0$, then
  \ $\Z{\gg}\ne\varnothing$.

\end{description}
\end{theorem}
Note that $\gg$ has $1$-dimensional ideal if its center
is nontrivial, or $\gg$ is finite dimensional and supersoluble
(Jacobson \cite[Ch{.\ }2, Th{.\ }14]{Jacobson62}).  A finite
dimensional solvable Lie algebra of vector fields on a surface has
derived length $\le 3$ (Epstein \& Thurston \cite{ET79}).  Plante
\cite{Plante88} points out that $\V^\omega (\R 2)$ contains
infinite-dimensional subalgebras, such as the Lie algebra of quadratic
vector fields.

An {\em action} of a group $G$ on a manifold $M$ is a homomorphism
$\alpha\co g\to g^\alpha$  from $G$ to the homeomorphism group of $P$,
such that the 
corresponding {\em evaluation map} 
\[
 \msf{ev_\alpha}\co G\times P\to P, \quad (g,p)\mapsto g^\alpha (p)
\]
is continuous.  When 
$\msf{ev_\alpha}$ is  analytic, $\alpha$ is an {\em analytic action}. 
\begin{theorem}         \mylabel{th:liegroup}
Assume $M$ is a compact analytic surface and  $G$
is  a connected Lie group having a one-dimensional normal subgroup.
If $\chi (M)\ne 0$, every effective analytic action of $G$ on $M$ has a fixed
point.
\end{theorem}
For supersoluble $G$ this is due to Hirsch \& Weinstein \cite{HW00}.
\subsection{Background on group actions}   
The literature on actions of connected
Lie groups $G$ include the following notable results:
\begin{proposition}             \mylabel{th:ET}
If $G$ is solvable (respectively, nilpotent) and acts
  effectively on an $n$-dimensional manifold,  the derived length of
  $G$ is $\le n+1$   (respectively, $\le n$) \

{\em (Epstein \& Thurston \cite{ET79}).} 
\end{proposition}

In the next two propositions $M$ denotes  a compact connected surface.  
\begin{proposition}             \mylabel{th:background}
Assume $G$ acts on  $M$ without fixed points. 
\begin{description}

\item[(i)]   If  $G$ is  nilpotent, $\chi (M)\ne 0$  \ 
{\em (Plante  \cite{Plante86}).}

\item[(ii)]  If the action is analytic, $\chi (M)\ge 0$  \ 
{\em (Turiel \cite{Turiel03}, Hirsch
\cite{Hirsch03}).}
\end{description}
\end{proposition}
%


%
\begin{proposition}    \mylabel{th:aff}
  Let ${\rm Aff}_+ (\R m)$ denote the group of
  orientation-preserving affine homeomorphisms of $\R m$. 
\begin{description}

\item[(a)] If $\chi (M) < 0$ and $G$ acts effectively on $M$ without
  fixed points,  then $G$
  has a quotient   isomorphic to  ${\rm Aff}_+ (\R 1)$ \ {\em
    (Belliart \cite{Belliart97}).} 

\item[(b)]  ${\rm Aff}_+(\R 1)$ has effective fixed-point free
  actions on $M$  \ {\em (Plante \cite{Plante86})}. 

\item[(c)]   
 ${\rm Aff}_+ (\R 2)$ has   effective analytic 
  actions on $M$  \ {\em (Turiel \cite{Turiel03})}.

\end{description}
\end{proposition}
For related results see the references above, also Belliart
\cite{Belliart97}, Hirsch \cite{Hirsch2010, Hirsch2011}, Molino \&
Turiel \cite{MolinoTuriel86, MolinoTuriel88}, Plante \cite{Plante91},
Thurston \cite{Thurston74}, Turiel \cite{Turiel89}.  Transitive
effective surface actions are classified in 
Mostow's thesis \cite{Mostow50}, with a useful summary in Belliart
\cite{Belliart97}.

\section{Dynamics}    \mylabel{sec:dynamics}
Let $\Psi:=\{\Psi_t\}_{t\in \msf T}$ denote a local flow ($\msf
T=\RR)$ or a local semiflow ($\msf T=\Rp)$ on a topological space $S$.
Each $\Psi_t$ is a homeomorphism from an open set $\mcal D
(\Psi_t)\subset S$
onto a set $\mcal R (\Psi_t)\subset S$, such that:

\begin{itemize}

\item $\Psi_t (p)$ is continuous in $(t,p)$,

\item $\Psi_0$ is the identity map of $S$,

\item if $0 \le |s| \le |t|$ and $|st|\ge 0$ then $\mcal D
  (\Psi^s)\supset \mcal D (\Psi^t)$,

\item  $  \Psi_t (\Psi_s (p)) =  \Psi_{t+s} (p)$.
\end{itemize}
We adopt the convention that notation of the form ``$\Psi_t (x)$''
presumes $x\in\mcal D (\Psi_t)$.

\begin{definition}              \mylabel{th:posinv}
 A set $L\subset S$ is {\em positively invariant} under $\Psi$
 provided $\Psi_t$ 
 maps $L\cap \mcal D 
 (\Psi_t)$ into $L$ for all $t\ge 0$, and {\em invariant} when
 $L\subset \mcal D (\Psi_t)\cap\mcal R (\Psi_t)$ for all $t  \in \msf
 T$.  When $\Psi$ is generated by  a vector field $Y$ we use the
 analogous terms  ``positively $Y$-invariant'' and ``$Y$-invariant.'' 

\end{definition}

Let $M, \tilde M$ be as in Hypothesis \ref{th:hypmain}. 
When $\Psi$ is a local semiflow on $M$, the theorem on
invariance of domain shows that $\mcal R (\Psi_t)$ is open in $M$ when
$\Phi$ is a local flow, and also when $\mcal D (\Psi_t)\cap\p M=\empty$.
This implies
 $M\setminus \p M$ is  positively
invariant under every local semiflow on $M$. 

\begin{proposition}             \mylabel{th:wedge}
Assume  $X, Y\in \V_{\msf{in}} (M)$ and $[X, Y]\wedge X=0$.
\begin{description}

\item[(a)] If $\Phi^Y_t (p)=q$ then $T_p\Phi^Y_t \co X_p\mapsto
  cX_q,\, c >0$.

\item[(b)]  $\Phi^Y_t$ sends integral curves of $X$  to
integral curves of $X$.

\item[(c)] $\Z X$ is $Y$-invariant.

\end{description}
\end{proposition}
\begin{proof} 
Let $\tilde X$ and $\tilde Y\in\V^1 (\tilde M)$ be
   extensions of   $X$ and $Y$, respectively.   It suffices to prove: 

\paragraph{(*)} 
{\em  If
 $p\in M,\ t\ge 0$ and  $p(t):= \Phi_t^{\tilde X} (p)$, then the linear map
\[T_p\Phi_t^{ \tilde Y}\co T_p \tilde M\to T_{p(t)}\tilde M\] sends $X_p$
to a positive scalar multiple of $X_{p(t)}$.}

\medskip
 By continuity it suffices to prove this  when $Y_p\ne 0$, $X_p\ne 0$
 and $|t|$ sufficiently small.  Working in flowbox coordinates $(u_j)$
 for $\tilde Y$ in a neighborhood of $p$, we assume $\tilde M$ is an
 open set in $ \R n$, \, $\tilde Y=\pde{~}{u_1}$, and $\tilde X$ has no zeros.
 Because $[X, Y]\wedge X=0$, there is a unique continuous map $f\co
 M\to \RR$ such that $[X, Y]=fX$.  Since $\tilde Y$ is a constant
 vector field, the
 vector-valued function $t\mapsto X_{p(t)}$ satisfies
\[
\ode {X_{p(t)}} t  = -f(p(t))\cdot X_{p(t)},
\]
whose solution is 
\[
X_{p (t)}= e^{-\int_0^t f(s)ds}\cdot X_{p(0)}.
\]
This  implies (*).
\end{proof}

The
 following fact is somewhat surprising because $T^{\msf{in}}_pM$ need
 not be convex in $T_pM$:  

\begin{proposition}           \mylabel{th:convex}
$\V^{\msf{L}}_{\msf{in}}(M)$ is a convex cone in $\V  (M)$.
\end{proposition}
\begin{proof} 
As $\V^{\msf{L}}(M)$  is a convex cone in $\V(M)$, it
  suffices to show that $\V^{\msf{L}}_{\msf{in}} (M)$ is 
  closed under addition.  Let  $X, Y\in \V^{\msf{L}}(M)$.  
We need  to prove: 
\begin{equation}                \label{eq:ppM}
\mbox{\em If $p\in \p M$ there exists $\eps >0$ such that $0\le t\le \eps
  \implies \Phi_t^{ X + Y} (p)\in M$.}
\end{equation}
This is easily reduced to a local result, hence we assume $M$ is
relatively open in the closed halfplane $\RR \times [0,\infty)$ and
  $X, Y$ are Lipschitz vector rields on $M$.  Let $\tilde X, \tilde Y$
  be extensions of $X, Y$ to Lipschitz vector fields on an open
  neighborhood $\tilde M\subset \R 2$ of $M$ ( Johnson {\em et
    al.\ }\cite {Lindenstrauss86}).  Denote the local flows of $\tilde
  X, \, \tilde Y, \, \tilde X+ \tilde Y$ respectively by $\{f_t\},
  \{g_t\}, \{h_t\}, \ (t\in \RR)$.  We use a special case of Nelson
  \cite[Th.\,1, Sec.\,4]{Nelson70}:
\begin{proposition}             \label{th:nelson}
For every $p\in \tilde M$ there exists $\eps >0$ and a neighborhood $W\subset
\tilde M$
of $p$  such that
\[h_t (x)=\lim_{k\to\infty}\left(
f_{t/k}\circ g_{t/k}\right)^k (x)
\]
  uniformly for $x\in W$ and  $|t|<\eps$.
\end{proposition}
Because $X$ and $Y$ are inward, $M$ is positively invariant under the
local semiflows $\{f_t\}_{t\ge 0}$ and $\{ g_t\}_{t\ge 0}$.  Therefore
\[
0\le t\le \eps \implies \left(f_{t/k}\circ g_{t/k}\right)^k \in M,
\qquad (k\in \Np).
\]
As $M$ is relatively  closed in $\tilde M$, Proposition
\ref{th:nelson} implies $h_t (x)\in M$ for $0\le t\le \eps$, 
which yields (\ref{eq:ppM}).
\end{proof}

Examination of the proof yields:
\begin{corollary}               \mylabel{th:convexcor}
If $L$ is a closed subset of a smooth manifold $N$, the set of locally
Lipschitz vector fields on $N$ for which $L$ is positively invariant
is a convex cone. \qed
\end{corollary}

\begin{question} Is $\V_{\msf{in}}(M)$ is a convex cone in $\V (M)$?
\end{question}

\section{Index functions}   \mylabel{sec:index}
 We review properties of the  fixed point index $I (f)$
 defined  by the late Professor Albrecht Dold (\cite {Dold65, Dold72}).
 Using it we define an {\em equilibrium index} $I_K (\Phi)$ for local semiflows,
 and a {\em vector field index} $\msf i_K (X)$ for inward vector
 fields.  
\subsection{The fixed point index for maps}   \mylabel{sec:fpindex}

\paragraph{Dold's Hypothesis:}
\begin{itemize}
\item $V$ is an open set in a topological space $S$.

\item  $f\co V \to
 S$ is a continuous map with compact fixed point set \,$\Fix
f\subset V$.  

\item $V$ is a  Euclidean neighborhood retract (ENR).\footnote{
This means $V$ is homeomorphic to a retract of an open subset of
  some Euclidean space.   Polyhedra and connected metrizable 
  manifolds are ENRs.}
\end{itemize}

On the class  of maps satisfying these conditions, Dold
 constructs an integer-valued {\em fixed point index}
 denoted here by $I (f)$, uniquely characterized by the
following five properties (see \cite [VII.5.17, Ex.\,5*]{Dold72}):
\begin{description} 

\item[(FP1)]  $I(f)=I(f|V_0)$\, if $V_0\subset V$ is an open
  neighborhood of   $\Fix f$.  

\item[(FP2)]    $I(f)=\begin{cases}
& 0 \ \mbox{if   $\Fix f =\empty$,}\\
& 1 \ \mbox{if   $f$ is constant.}
\end{cases}
$

\item[(FP3)]    $I(f)=\sum_{i=1}^m I(f| V_i)$\, if $V$ is the disjoint
  union of $m$ open sets $V_i$.

\item[(FP4)]  
   $I (f\times g)=I (f)\cdot I (g)$.

\item[(FP5)]    $I(f_0)=I(f_1)$\, if there is a homotopy
  $f_t\co V\to S,\, (0\le t \le 1)$\, such that $\bigcup_t\Fix{f_t}$ is
  compact.
\end{description}
These correspond to (5.5.11) ---
  (5.5.15) in \cite[Chap.\,VII]{Dold72}.   

In addition:
\begin{description} 
\item[(FP6)] If $f$ is $C^1$ and $\Fix f$ is an isolated   fixed
  point $p$, then 
\[I(f)= (-1)^\nu
\]
where $\nu$ is the number of eigenvalues $\lam$ of $f$ such that $\lam
>1$, ignoring multiplicities  \,(\cite[VII.5.17, Ex. 3]{Dold72}).  
\end{description} 

\begin{description}

\item[(FP7)]  
If $S$ is an ENR  and $f\co S\to S$ is homotopic to the
identity map, then
\[
 I(f)= \chi (S).
\]
See Dold \cite[VII.6.22]{Dold72}.
\end{description} 

%
\begin{lemma}           \mylabel{th:gnearf}
If\, $g$\, is sufficiently close to $f$ in the compact
open topology, then $\Fix g$ is compact and $I(g)=I(f)$. 
\end{lemma}
\begin{proof}  We can assume $\rho\co W\to V$ is a retraction, where
  $W\subset \R   m$ is an open set containing $V$.  For $g$ sufficiently close to $f$ the following hold:
  $W$ contains the line segment (or point) spanned by $\{f(p), g(p)\}$
  for every $p\in V$,   and the maps 
\[
f_t\co (t,p)\mapsto \rho((1-t)f(p) + t g(p)),\qquad (0\le t\le 1,\quad p\in V)
\]
constitute a homotopy in $V$
from $f_0=f$ to $f_1=g$ with $\bigcup_t\Fix{f_t}$ is
  compact.  Therefore the conclusion follows from  (FP5).
\end{proof}

\subsection{The equilibrium index for local semiflows}
Let $\Phi:=\{\Phi_t\}_{t\ge 0}$ be a local semiflow in a topological
space $\mcal C$, with {\em equilibrium set}
\[
\mcal E (\Phi):={\textstyle\bigcap}_{t\ge 0}\Fix{\Phi_t}. 
\] 
 $K\subset \mcal E (\Phi)$ is a {\em block} if $K$ is compact and
has an open, precompact ENR neighborhood $V\subset \mcal C$ such that $\ov
V\cap\mcal E (\Phi)\subset V$.  Such a $V$ is an {\em
  isolating neighborhood} for $K$.  With these assumptions on $\Phi$
and $V$ we have:
\begin{lemma}           \mylabel{th:tau}
There exists $\tau >0$ such that the
following hold when $0  <t \le \tau$:

\begin{description}

\item[(a)] $\Fix {\Phi_t}\cap V$ is  compact,

\item[(b)] $I (\Phi_t|V)=I (\Phi_\tau|V)$.

\end{description}
\end{lemma}
\begin{proof}
If (a) fails, there are convergent sequences $\{t_k\}$ in  $[0,\infty)$
and  $\{p_k\}$ in $V$   such that
\[
 t_k\searrow 0, \quad
p_k\in \Fix{\Phi_{t_k}}\cap V, \quad p_k\to q\in \msf{bd} (V).
\]       
Joint continuity of $(t,x)\mapsto \Phi_t (x)$ yields the contradiction
$q\in \mcal E (\Phi) \cap \msf{bd} (V)$.  Assertion (b) is a
consequence of (a) and (FP5).
\end{proof}

It follows that the fixed point index $I
(\Phi_\tau|V)$ depends only on $\Phi$ and $K$, and is the same for all
isolating neighborhoods $V$ of $K$. 
\begin{definition}              \mylabel{th:defeqindex}
Let $\tau>0$ be as in Lemma \ref{th:tau}(b).  We call $I(\Phi_\tau|V)$
the {\em equilibrium index} of $\Phi$ in $V$, and at $K$, denoted by
$\msf i(\Phi,V)$ and $\msf i_K (\Phi)$.
\end{definition}

\subsection{The vector field index for inward vector fields} 
In the rest of this section the manifolds $\tilde M$ and $M\subset
\tilde M$ are as in Hypothesis \ref{th:hypmain}, $K$ is a block of
zeros for $X$ for $X\in\V_{\msf{in}} (M)$, \ and $U$ an isolating
neighborhood for $(X, K)$.  Then $K$ is also a block of equilibria for
the local semiflow $\Phi^X$, and the equilibrium index $\msf i (\Phi^X,U)$
is defined in Definition \ref{th:defeqindex}.
\begin{definition}             \mylabel{th:defindex}
The {\em vector field index of $X$ in $U$} (and\, {\em at $K$})  is  
\[\begin{split}
 \msf i (X, U)=\msf i_K (X) & :=  \msf i (\Phi^X, U).
\end{split}
\]
$K$ is {\em essential} (for $X$) \ if $\msf i_K (X)\ne 0$, and {\em
  inessential} otherwise. 
\end{definition}

Two vector fields $X_j\in \V_{\msf{in}} (M_j), \,j=1,2$ have {\em
  isomorphic germs} 
at $K_j\subset M_j$ provided there are open neighborhoods $U_j\subset
M_j$ of $X_j$ and  a homeomorphism $U_1\approx U_2$ conjugating $\Phi^{X_1}|U_1
$ to $\Phi^{X_2}|U_2$.

\begin{proposition}             \mylabel{th:index}
The vector field index has the following properties:
\begin{description}

\item[(VF1)] $i_K (X)=\msf i_{K'} (X')$\, if $K'$ is a block of zeros
  for $X'\in \V_{\msf{in}} (M')$ and the germs of $X$ at $K$ and $X'$
  at $K'$ are isomomorphic.

\item[(VF2)]   If $\msf i (X, U)\ne 0$ then $\Z X\cap
  U\ne\varnothing$.

\item[(VF3)] $\msf i (X, U)=\sum_{j=1}^m \msf i (X, U_j)$ provided $U$ is the
  union of disjoint  open sets $U_1,\dots,U_m$. 

\item[(VF4)] If $Y$ is sufficiently close to $X$ in $\V_{\msf{in}}
  (M)$, then $U$ is isolating for $Y$ and $\msf i (X, U)=\msf i
  (Y, U)$.

\item[(VF5)] $\msf i_K (X)$  equals the
Poincar\'e-Hopf index  $\msf i^{\rm PH}_K (X)$ provided  $K\cap\p M=\empty$. 

\end{description}
\end{proposition}
\begin{proof}
{~}

{(VF1): }   A consequence of  (FP1).

{(VF2): }   Follows from (FP2).

{(VF3): } Follows from  (FP3).

{(VF4): } Use Lemma \ref{th:gnearf}. 

{(VF5): } Since $X$ can be approximated by locally $C^\infty$ vector fields
transverse to the zero section,  using compactness of $\ov U$ and (VF4) we
assume $\Z X \cap U$ is finite set of hyperbolic equilibria. 
 By
(FP3) we assume $\Z X \cap
U$ is a hyperbolic equilibrium $p$. In this case the
index of $X$ at $p$ is $(-1)^\nu$ where $\nu$ is the
number of positive eigenvalues of $dX_p$ (ignoring
multiplicity).  The conclusion follows from  (FP6) and Definitions
\ref{th:defeqindex},   \ref{th:defindex}.
\end{proof}
 
\begin{proposition}             \mylabel{th:ichi2}
If $M$ is compact, \,$\msf i (X, M)= \chi (M)$.
\end{proposition}
\begin{proof} 
Follows from (FP7).
\end{proof}

\begin{proposition}             \mylabel{th:homsections}
Assume   $X, Y\in \V_{\msf {in}} (M)$  and $U\subset M$ is isolating
for $X$.    Then   
$U$ is isolating for $Y$,  and 
\[
\msf i (X, U) =\msf i
  (Y, U),
\]
 provided  one of the following holds: 
\begin{description}

\item[(i)]   $Y|\msf{bd}(U)$ is  sufficiently close to $X|\msf {bd}  (U)$,

\item[(ii)]
$Y|\msf{bd}(U)$ is nonsingularly  homotopic to $X|\msf{bd}(U)$.
\end{description}
\end{proposition}
\begin{proof} Both (i) and (ii) imply $U$ is isolating for $Y$.  
 Consider the homotopy 
\[
 Z^t:= (1-t)X + tY,\qquad (0\le t\le 1).
\]
When  (i) holds each vector field
 $Z^t|\bd U$ is nonsingular, implying 
 (ii).  In addition, $Z^t$
is  inward by Proposition
\ref{th:convex}, and Lemma  \ref{th:tau} yields $\tau >0$ such that
\[  
 0<t\le \tau \implies 
   \msf i (X, U)=I(\Phi^X_t|U), \quad  \msf i (Y,
   U)=I(\Phi^Y_t|U). 
\] 
By 
Lemma \ref{th:gnearf}, each $t\in [0,1]$ has an open
neighborhood $J_t\subset[0,1]$ 
such that 
\[
  s\in J_t\implies I (\Phi^{X^s}_\tau|U) =
 I (\Phi^{X^t}_\tau|U).  
\]
Covering $[0,1]$ with sets $J_{t_1},\dots, J_{t_\nu}$ and 
inducting on $\nu\in \Np$  shows that  
\[
  I (\Phi^X_\tau|U) =  I (\Phi^Y_\tau|U),
\]
which by Definition  \ref{th:defindex} implies the conclusion.
\end{proof}

\paragraph{The index as an obstruction}
The following algebraic calculation of the index is included for
completeness, but not used.  Assume $M$ is oriented and $U$ is an
isolating neighborhood for a block $K\subset\Z X$.  Let $V\subset U$
be a compact smooth $n$-manifold with the orientation induced from
$M$, such that $K\subset V\verb=\= \p V$.  The primary obstruction to
extending $X|\p V$ to a nonsingular section of $TV$ is the {\em
  relative Euler class}
\[
  {\bf  e}_{(X, V)}\in H^n (V, \p V)
\]   
Let
\[{\bf v}\in H_n (V, \p V)
\]
 be the  homology class corresponding to
the induced orientation of $V$. %
Denote by 
\[
  H^n (V, \p V) \times H_n (V,\p V)\to\ZZ,\quad ({\bf c}, {\bf u}) \to 
\langle {\bf c}, {\bf u}\rangle,
\]
the  Kronecker Index pairing,  induced by evaluating cocycles on cycles. 
Unwinding definitions leads to:
\begin{proposition}             \mylabel{th:obs}
With  $M, X, K, V$ are as above, 
\[
 \msf i_K (X)= \langle {\bf v}, {\bf e}_{(X, V)} \rangle. 
\]
\end{proposition}

When $M$ is nonorientable the same formula holds provided the
coefficients for $H^n (V, \p V)$ and $ H_n (V,\p V)$ are twisted by
the orientation sheaf of $V$.

\subsection{Stability of essential blocks }   \mylabel{sec:stability}
An immediate consequence of Propositions \ref{th:homsections}(i) and property
(VF2) of \ref{th:index} is: 
\begin{corollary}               \mylabel{th:ess}
If a block $K$ is essential for $X$, and $Y\in \V_{\msf{in}} (M)$ is
sufficiently close to $X$, then every neighborhood of $K$ contains an
essential block for $Y$.  \qed
\end{corollary}

Thus essential blocks are stable under perturbations of the vector
field.   It is easy to see that a block is stable if it contains a
stable block.  For example,  the block
$\{-1, 1\}$ for $X= (x^2-1)\pd x$ on $\RR$  is stable,  but
inessential.   But the following result 
(not used) means that a block can be perturbed away  iff every subblock
is inessential:
\begin{proposition}             \mylabel{th:iness}
Assume $\p M$ is a smooth submanifold of $\tilde M$ and every block
for $X$ in $U$ is inessential.  Then
$ X=\lim_{k\to\infty}X^k $ where 
\[
 X=\lim_{k\to\infty}X^k, \quad X^k\in \V^{\msf L}_{\msf{in}} (M), \quad 
  \Z{X^k}\cap  U =\empty, \qquad (k\in \Np) 
\]
and $X^k$ coincides with $X$ outside $U$. 
\end{proposition}
\begin{proof}
Fix a Riemannian metric on $M$.  For every $\eps >0$ choose an
isolating neighborhood $W:=W(\eps)\subset U$ of $K$ having only
finitely many components, and such that
\[
\|X_p\|<\eps, \qquad (p\in W). 
\]
Thus $X (W)$ lies in the bundle $T^\eps W$ whose fibre over $p$ 
is the open disk of radius $\eps$ in $T_p W$.  
Smoothness of $\p M$ enables an approximation  $Y^\eps\in \V^{\msf
  L}_{\msf{in}}$ to $X$ such that $Y^\eps
(W)\subset  T^\eps W$ and $\Z {Y^\eps}\cap \ov U$ is finite.  
By Proposition \ref{th:homsections}(ii) and the hypothesisw we choose
the approximation 
close enough so that for each
component $W_j$ of $W$:
\[
\msf i (Y^\eps, W_j) =0.
\]
Standard deformation techniques (compare Hirsch \cite[Th. 5.2.10]{Hirsch76})
permit pairwise cancellation in each $W_j$ of the zeros of $Y^\eps$,
without changing $Y^\eps$ near $\msf{bd} (W_j)$.  This yields a vector
field $X^{\eps}\in \V^{\msf L}_{\msf{in}} (M)$ coinciding with $X$ in a
neighborhood of $M\verb=\= W$ and nonsingular in $W$, and such that
\[
\|X^\eps_p - X_p\| < 2\eps, \qquad (p\in M)
 \]
The sequence $\{X^{1/k}\}_{k\in\Np}$ has the required properties. 
\end{proof}

Plante \cite{Plante91} discusses  index functions  for abelian
Lie  algebras of vector fields on closed surfaces. 

\section{Proofs of the main theorems }   \mylabel{sec:mainproofs}

\subsection{Proof of Theorem \ref{th:main}}   \mylabel{sec:proofs}
We recall the hypothesis:

\begin{itemize} 

\item $\tilde M$ is an analytic surface with empty boundary,
  $M\subset \tilde M$ is a connected topological surface embedded in
  $\tilde M$.

\item $X$ and $Y$ are inward $C^1$ vector fields on $M$.

\item   $[X, Y]\wedge X=0$.  

\item $K$ is an essential block of zeros for $X$.    

\end{itemize}
\begin{definition}              \mylabel{th:defdepend}
Let $A, B$ be vector fields on a set $S\subset \tilde M$. 
The {\em dependency set} of $A$ and $B$ is 
\[
  \msf D (A, B):=\{p\in S\co A_p\wedge B_p=0\}
\]
\end{definition}
Evidently
\[
  \msf D (X, Y)=   \msf D (\tilde X, \tilde Y)\cap M.
\]
Proposition
\ref{th:wedge} implies $ \msf D (\tilde X, \tilde Y)$ is $\tilde
Y$-invariant and $\msf D (X, Y)$ is positively $Y$-invariant. 

\medskip
{\em Case (a):     $X$ and $Y$ are analytic.}   Then
$\msf D (\tilde X, \tilde Y)$ and its subset $\Z{\tilde X}$ are
analytic sets in $\tilde M$, hence $\tilde M$ is a simplicial
complex with  subcomplexes $\msf D (\tilde X, \tilde Y)$ and $\Z{\tilde
X}$ by S.\ {\L}ojasiewicz's  triangulation theorem \cite{Lo64}.

We assume 
\begin{itemize}
\item $\dim \Z {\tilde Y} < 2$, 
\end{itemize}
as otherwise $Y=0$ and the conclusion is trivial.  We also assume
\begin{itemize}
\item {\em every component of $K$ has dimension $1$}
\end{itemize}  
because isolated points of $K$ lie in $\Z Y$ and $K=M$ by analyticity
if some component of $K$ is $2$-dimensional, and either of these
conditions imply the conclusion.  Thus $\Psi^{\tilde Y}$ restricts to
a semiflow on the $1$-dimensional complex $ D (\tilde X, \tilde Y)$
having $\Z X$ and $\msf D (X, Y)$ as positively invariant
subcomplexes.

Let $J\subset K$ be any component.  $J$ is a compact, connected,
triangulable space of dimension $\le 1$ which is positively
$Y$-invariant.  From the topology of $J$ we see that $\Z Y$ meets $J$
and therefore $K$, unless
\begin{equation}                \label{eq:jordan}
\mbox {\em $J$ is a Jordan curve on which $\Phi^Y$ acts
  transitively.} 
\end{equation}
Henceforth (\ref{eq:jordan}) is assumed.   

Let $L\subset \msf D (X, Y)$ 
be the component containing $J$.  The set $Q:=J\cap \ov{L\setminus J}$ is
positively $Y$-invariant, whence  (\ref{eq:jordan}) implies
$Q=J$ or $Q=\empty$.  Therefore one of the following holds:

\begin{description}

\item[(D1)] $J\subset \Int_M\msf D (X, Y)$, 

\item[(D2)] $J$ is a component of  $\msf D (X, Y)$. 

\end{description}
 
Assume (D1) and suppose {\em per contra} that $\Z Y\cap K
=\varnothing$.  Then $ \msf D (X, Y)$ contains the compact closure of
an open set $U$ that is isolating for $(X,K)$.  We choose $U$ so that                 
each component $C$ of the topological boundary $\msf{bd}(U)$ is a
Jordan curve or a compact arc.  It suffices by Proposition
\ref{th:homsections}(ii) to prove for each $C$:
\begin{equation}                \label{eq:scyc}
\mbox{\em  the vector fields $X|C$ and $Y|C$ are nonsingularly homotopic.}
\end{equation}
Since this holds when $C$ is an arc, we assume $C$ is a Jordan curve.   Fix
 a Riemannian metric on $M$ and define
\[
 \hat X_p=\frac{1}{\|X_p\|}X_p, \quad \hat Y_p=\frac{1}{\|Y_p\|}Y_p,
 \qquad (p\in C).
\]
These  unit vector fields  are nonsingularly homotopic to $X|C$ and
$Y|C$ respectively, and the assumption $C\subset \msf D (X, Y)$
implies $\hat X=\hat Y$ or $\hat X= -\hat Y$.  In the first case there
is nothing more to prove.  In the second case $\hat X$ and $\hat Y$ are
antipodal sections of the unit circle bundle $\eta$ associated to
$T_CM$.  As the identity and antipodal maps of the circle are
homotopic through rotations, (\ref{eq:scyc}) is proved.

Now assume (D2). There is  an isolating neighborhood $U$ for $X$ such that 
\begin{equation}                \label{eq:UD}
U\cap \msf D (X, Y) = K.
\end{equation}
If  $0<\eps <1$ the field $X^\eps:= (1-\eps X)+\eps
Y$  belongs to $\V^{\msf L}_{\msf{in}}(M)$ (Proposition \ref{th:convex}), and
has a zero $p\in U$ provided $\eps$ is sufficiently small 
(Proposition \ref{th:ess}).  In that case $X_p$ and $Y_p$ are linearly
dependent, therefore  $p\in K$ by (\ref{eq:UD}), whence  $Y_p=0$.

\medskip

{\em Case (b):  Every neighborhood of $K$ contains an open neighborhood $W$
  whose boundary consist of finitely many  $Y$-cycles.}  \
It suffices to prove that $\Z Y\cap \ov W \ne\varnothing$ if  $W$ is
isolating for $(X, K)$.  Given such a $W$, let  $C$ be a component of
$\msf{bd}(W)$.  By Proposition \ref{th:wedge}(a), $X_p$ and $Y_p$ are
linearly dependent at all points of $C$, or at no point of $C$.  In
the first case $X|C$ and $Y|C$ are nonsingularly homotopic, as in the
proof of (\ref{eq:scyc}).  In the second case they are nonsingularly
homotopic by the restriction to $C$ of the path of vector fields
$(1-t)X+tY$, $0\le t\le 1$.  It follows 
that 
$X|\msf{bd}(W)$ and $Y|\msf{bd}(W)$ are nonsingularly homotopic.  Now  
Proposition \ref{th:homsections}(ii) implies 
\[
 \msf i(Y, W)=\msf i (X, W),
\]
which is nonzero because $K$ is essential for $X$.  Hence 
either $\Z Y$ meets $\msf {bd} (W)$, or $W$ is isolating for $(Y, K)$ and
$\Z Y\cap W\ne\varnothing$.  This shows that $\Z Y$ meets $\ov W$. \qed

\begin{remark}          \mylabel{th:remlima}
It is interesting to see where the proof Theorem \ref{th:main} breaks
down in Lima's counterexample to a nonanalytic version (see Example
\ref{th:exlima}).  Lima starts from the planar vector fields
\[
X^1:=  \pd x,\qquad Y^1:=x \pd  x +  y \pd y 
\]
satisfying $[X^1, Y^1]= X^1$ and transfers them to the open disk by an
analytic diffeomormophism $f\co \R 2 \approx \Int \D 2$.  This is done
in such a way that the push-forwards of $X^1$ and $Y^1$ extend to
continuous vector fields $X, Y$ on $M:=\D 2$ satisfying $[X, Y]= Y$,
with $\Z X=K= \p \D 2$ and $\msf i_K (X) =1$, while $\Z Y$ is a
singleton in the interior of $\D 2$.  This can be done so that $X$ and
$Y$ are $C^\infty$ (see \cite{BL96}) and therefore generate unique
local semiflows.  The dependency set $\msf D (X, Y)$ is $R\cup\p \D
2$, where $R$ is the $\Phi^X$-orbit of $z$, a topological line in
$\Int\, \D 2$ that spirals toward the boundary in both direction.
$\msf D (X, Y)$ is not triangulable because it is connected but not
path connected.  It is easily seen that neither (D1) nor (D2)
holds.
\end{remark} 

\subsection{Proof of Theorem \ref{th:mainB}}   \mylabel{sec:proofmainB}
Here $K$ is essential for $X$ and $\Phi^Y$
preserves area.  
Suppose {\em per contra}  
\begin{equation}                \label{eq:zyk}
\Z Y\cap K=\empty. 
\end{equation}
We can assume $K$ contains a $Y$-cycle $\gam$, for (\ref{eq:zyk})
implies every minimal set for $\Phi^Y$ in $K$ is a cycle: This follows
from the Poincar\'e-Bendixson theorem (Hartman \cite{Hartman64}) when
$K$ has a planar neighborhood, and from the Schwartz-Sacksteder
Theorem \cite{Sacksteder65,Schwartz63} when $Y$ is $C^2$.

Let $J\subset M$ be a half-open arc with endpoint $p\in \gam$ and
otherwise topologically transverse to $Y$ orbits (Whitney
\cite{Whitney33}).  For any sufficiently small half-open subarc
$J_0\subset J$ with endpoint $p$, there is a first-return Poincar\'e
map $f\co J_0\hookrightarrow J$ obtained by following trajectories.
By the area-preserving hypothesis and Fubini's Theorem, $f$ is the
identity map of $J_0$.  Therefore $\gam$ has a neighborhood $U\subset
M$, homeomorphic to a cylinder or a M\"obius band, filled with
$Y$-cycles.  Theorem
\ref{th:main}(b) implies  $K$ is inessential for
$X$, contradicting to the hypothesis.  \qed

\subsection{Proof of Theorem \ref{th:mainC}}  
$K$ is an essential block for $X\in\V^\om_{\msf{in}}(M)$ and
$\p M$ is an analytic set in $\tilde M$.  We can assume $\msf W
(X)\ne\varnothing$ (see Definition  \ref{th:defW}).   Our goal is
to prove   
\[
 \Z{\msf W
  (X)}\cap K\ne\varnothing.
\]
The main step is to show that the set 
\[
  \mcal P (K):=\{\ttt\subset \msf W  (X)\co \Z\ttt \cap
  K\ne\varnothing\}
\] 
is inductively ordered by inclusion. 
Note that $\mcal P (K)$ is nonempty because it contains the singleton
$\{Y\}$.  In fact Theorem \ref{th:main} states:
\begin{equation}                \label{eq:yddx}
Y\in \msf W (X)\implies  \{Y\}\in  \mcal P (K).
\end{equation}

We rely on a consequence of Proposition \ref{th:wedge}:
\begin{equation}                \label{eq:kddx}
\mbox{\em $K$ is positively invariant under $\msf W (X)$}
\end{equation}

The assumption on $\p M$ implies $M$ is  semianalytic as a subset of 
$\tilde M$, and this implies $K$ is also semianalytic.  Therefore $K$,
being compact, has only finitely many components and
one of them is essential for $X$ by Proposition \ref{th:index}, (VF3).
Therefore we can assume
$K$ is a connected semianalytic set of $\tilde M$.
Now $K$ is the intersection of $M$ with the component of $\Z \tilde X$
that contains $K$, which is an analytic set.  This implies $\dim K \le
1$, and we assume $\dim K=1$, as otherwise $K$ is finite and contained
in $\Z{\msf W (X)}$ by (\ref{eq:kddx}).

The set $K_{sing}\subset K$ where $K$ is not locally an
analytic $1$-manifold is finite and  positively invariant under  $\msf W
(X)$   As
this implies  $K_{sing}\subset\Z {\msf W (X)}\cap K$, 
we can assume $K_{sing}=\empty$, which under current assumptions means:
\[
  \mbox{\em  $K$ is an analytic submanifold of $\tilde M$
  diffeomorphic to a circle.}  
\]
We can also assume: 
\begin{description}
\item[(K)]
{\em $K\subset L$ if $K\cap L\ne\varnothing$ and $L$ is positively
  $\msf W (X)$-invariant 
semianalytic set in $\tilde M$.}
\end{description}
For if  $Y\in \msf W (S)$ then
$K\cap L$, being  nonempty,  finite   and positively invariant under
every $Y$, is 
necessarily contained in $\Z Y$. 

From (K) we infer 
 \begin{equation}                \label{eq:hhK}
\ttt\in \mcal P (K) \implies K\subset \Z\ttt.
  \end{equation}
Consequently  $\mcal P (K)$ is inductively ordered by inclusion.   By 
Zorn's lemma there is 
a maximal element $\mm\in\mcal P (K)$, and (\ref{eq:hhK}) implies
 \begin{equation}                \label{eq:kmm}
 K\subset \Z \mm.
 \end{equation}
To  prove  every  $Y\in \gg$ lies in $\mm$, let $\nn_Y\subset\gg$ be
the smallest ideal 
containing $Y$ and $\mm$.  Theorem \ref{th:main} implies $\Z Y\cap
K\ne\varnothing$, whence
\[
 \Z {\nn_Y}\cap K=  \Z Y\cap \Z \mm \cap K \ne\varnothing
\]
 by (\ref{eq:kmm}).  Property (K) shows that $\nn_Y\in \mcal
 H$, so $\nn_Y=\mm$  by  maximality of $\mm$.  \qed

\subsection{Proof of Theorem \ref{th:liealg} }
The theorem states:

\smallskip
\noindent
{\em Let $M$ be an analytic surface and $\gg$ a  Lie
  algebra  of analytic vector fields on $M$ that are 
tangent to $\p M$.  If $X\in \gg$ spans a one-dimensional ideal,
then:}

\begin{description}

\item[(a)] {\em $\Z \gg$ meets every
essential block $K$ of zeros for $X$,}

\item[(b)] {\em  if $M$ is compact and $\chi (M)\ne 0$ then
  $\Z{\gg}\ne\varnothing$.} 

\end{description}

The hypotheses imply $\gg\subset \msf W (X)$, because if $Y\in \gg$ then
$[X, Y]=cX, c\in \RR$.  Therefore (a) follows from Theorem
\ref{th:mainC}.  Conclusion (b) is a consequence, because its
assumptions imply the block $\Z X$ is essential for $X$ (Proposition
\ref{th:ichi2}).  \qed

\subsection{Proof  of Theorem \ref{th:liegroup}}  
An effective analytic action $\alpha$ of $G$ on $M$ induces an isomorphism
$\phi$ mapping the Lie algebra $\gg_0$ of $G$ isomorphically onto a
subalgebra $\gg\subset \V^\om (M)$.   Let $X^0\in\gg_0$ span the Lie
algebra of a one-dimensional normal subgroup of $G$.  Then $\phi
(X^0)$ spans a $1$-dimensional ideal in $\gg$, hence Theorem
\ref{th:liealg} implies $\Z \gg\ne\varnothing$.  The conclusion
follows because $\Z \gg=\Fix {\alpha(G)}$. \qed


\begin{thebibliography}{99}
%
\bibitem{Belliart97} M.\ Belliart, {\em Actions sans points fixes sur les
surfaces compactes,} Math.\ Zeit.\ {\bf 225} (1997) 453--465


\bibitem{BL96} M. Belliart \& I. Liousse, {\em Actions affines sur les
  surfaces,} Publications   IRMA,  Universite  de Lille,  
  {\bf 38} (1996) expos\'{e} X  

\bibitem{Bonatti92} C.\ Bonatti, {\em Champs de vecteurs analytiques
commutants, en dimension $3$ ou $4$: existence de z\'{e}ros communs,}
Bol.\ Soc.\ Brasil.\ Mat.\ (N.S.) {\bf 22} (1992) 215--247


\bibitem{Dold65} A. Dold, {\em Fixed point index and fixed point
  theorem for Euclidean neighborhood retracts,} Topology {\bf 4}
  (1965)  1–--8

\bibitem{Dold72} A. Dold, ``Lectures on Algebraic Topology,'' Die
  Grundlehren der matematischen Wissenschaften Bd. 52, second edition.
  Springer, New York (1972)

 \bibitem{ET79} D.\ Epstein \& W.\ Thurston, {\em Transformation groups
  and natural bundles}, Proc.\ London Math.\ Soc.\ {\bf 38} (1979)
  219--236

\bibitem{Freedman82} M. Freedman, {\em The topology of 4-dimensional
manifolds,}  J. Diff. Geometry {\bf 17} (1982) 357--453

\bibitem{Friedl07} S. Friedl, I. Hambleton, P. Melvin, P. Teichner,
{\em Non-smoothable four-manifolds with infinite
  cyclic fundamental group,} Int. Math. Res. Not. 2007, Vol. 2007,
article ID rnm031, 20 pages, doi:10.1093/imrn/rnm031 

\bibitem{Hartman64} P. Hartman, ``Ordinary Differential Equations,''
   John Wiley \& Sons, New York (1964)

\bibitem{Hirsch76} M.\ Hirsch, ``Differential Topology,'' Springer-Verlag, New
York, 1976

\bibitem{Hirsch03} M.\ Hirsch,  manuscript (2003) 

\bibitem{Hirsch2010}  M.\ Hirsch,
{\em Actions of Lie groups and Lie algebras on manifolds,}
``A Celebration of the Mathematical Legacy of Raoul Bott,'' 69-78.
Centre de Recherches Math\'{e}matiques, U. de Montr\'{e}al.
Proceedings \& Lecture Notes Vol. 50 (P. R. Kotiuga, ed.),
Amer. Math. Soc. Providence RI (2010)

\bibitem{Hirsch2011}  M.\ Hirsch,
{\em Smooth actions of Lie groups and Lie algebras on manifolds,}
J.  Fixed Point Th.  App.
{\bf 10} (2011) 219--232

\bibitem{HW00} M.\ Hirsch \& A.\ Weinstein, {\em Fixed points of
analytic actions of supersoluble Lie groups on compact surfaces,}
Ergod.\ Th.\ Dyn.\ Sys.\ {\bf 21} (2001) 1783--1787

\bibitem{Hopf25} H.\ Hopf, {\em Vektorfelder in Mannifaltigkeiten},
 Math. Annalen {\bf 95} (1925), 340--367

\bibitem{Jacobson62} N.\ Jacobson, ``Lie Algebras,'' John Wiley \&
  Sons, New York (1962)

\bibitem{Kirby12} R.\ Kirby, Personal communication (2012)

\bibitem{Lima64} E.\ Lima, {\em Common singularities of commuting vector
fields on $2$-manifolds}, Comment.\ Math.\ Helv.\ {\bf 39} (1964)
97--110

\bibitem{Lindenstrauss86}
W. Johnson,  J. Lindenstrauss,  G. Schechtman,
{\em Extensions of Lipschitz maps into Banach
  spaces,} Israel J. Math. {\bf  54} (1986) 129-–138


\bibitem{Lo64} S.\ {\L}ojasiewicz, {\em Triangulation of semi-analytic
sets},  Ann.\ Scuola Norm. Sup. Pisa (3)  {\bf 18} (1964) 449--474

\bibitem{MolinoTuriel86} P. Molino \& F.-J. Turiel, {\em 
Une observation sur les actions de ${\bf R}^p$ sur les vari\'et\'es
compactes de caract\'eristique non nulle,}
Comment. Math. Helv. {\bf 61} (1986), 370–-375

\bibitem{MolinoTuriel88} P. Molino \& F.-J. Turiel, {\em Dimension des
  orbites d’une action de ${\bf R}^p$ sur une vari\'{e}t\'e compacte,}
  Comment. Math. Helv. {\bf 63} (1988) 253–-258

\bibitem{Mostow50} G.\ Mostow, {\em The extensibility of local Lie groups
of transformations and groups on surfaces}, Ann.\ of Math.\ {\bf 52}
(1950) 606--636.\

\bibitem{Nelson70} E.\ Nelson, ``Topics in Dynamics I: Flows,''
  Princeton University Press (1970)

\bibitem{Plante86} J.\ Plante, {\em Fixed points of Lie group actions on
surfaces,} Erg.\ Th.\ Dyn.\ Sys.\ {\bf 6} (1986) 149--161

\bibitem{Plante88} J.\ Plante, {\em Lie algebras of vector fields
which vanish at a point,} J. London Math. Soc (2) {\bf  38} (1988) 379--384

\bibitem{Plante91} J.\ Plante, {\em  Elementary zeros of Lie
 algebras of vector fields,} Topology {\bf 30} (1991) 215--222

\bibitem{Poincare85} H.\ Poincar\'e, {\em Sur les courbes d\'efinies par une
\'equation diff\'erentielle,} J. Math. Pures Appl.  {\bf 1} (1885)
167--244

\bibitem{Schwartz63}  A. J. Schwartz, {\em Generalization of a
  Poincar\'e-Bendixson theorem to closed two dimensional manifolds,} 
 Amer. J. Math. {\bf 85} (1963) 453-458 

\bibitem{Sacksteder65} R.  Sacksteder, {\em Foliations and
  pseudogroups,} Amer. J. Math. {\bf 87} (1965) 79--102


\bibitem{Thurston74}  W. Thurston, {\em A generalization of the Reeb
  stability theorem}, Topology {\bf 13} (1974) 347-352

\bibitem{Turiel89} F.-J.\ Turiel, {\em An elementary proof of a Lima's
theorem for surfaces}, Publ.\ Mat.\ {\bf 3} (1989) 555--557

\bibitem{Turiel03} F.-J.\ Turiel,  {\em Analytic actions on compact
  surfaces and fixed points,} Manuscripta Mathematica {\bf 110} (2003)
  195--201
\bibitem{Turiel06} F.-J.\ Turiel, personal communication (2006)

\bibitem{Whitney33} H. Whitney, {\em Regular families of curves},
Ann. Math.  {\bf 34} (1933)  244--270

\end{thebibliography}
\end{document}